\documentclass[11pt]{amsart}
\usepackage{amssymb}
\usepackage[utf8]{inputenc}
\usepackage{enumerate}

\newcommand{\set}[1]{\left\lbrace #1\right\rbrace}
\providecommand{\abs}[1]{\left\lvert#1\right\rvert}

\newcommand{\qtq}[1]{\quad\text{#1}\quad}
\newcommand{\remove}[1]{ }

\newtheorem{theorem}{Theorem}[section]
\newtheorem{proposition}[theorem]{Proposition}
\newtheorem{lemma}[theorem]{Lemma}
\newtheorem{corollary}[theorem]{Corollary}

\newtheorem{remark}[theorem]{Remark}
\newtheorem{remarks}[theorem]{Remarks}
\newtheorem*{theorem*}{Theorem}\numberwithin{equation}{section}

\long\def\salta#1\finqui{\relax}                           
\usepackage{color}                                               
\long\def\verde#1\basta{ {\color{red}#1} }

\begin{document}

\title{Cantor type functions in non-integer bases} 
\author[C. Baiocchi]{Claudio Baiocchi} 
\email{bici.nando@gmail.com}
\author[V. Komornik]{Vilmos Komornik}
\remove{\address{16 rue de Copenhague\\
         67000 Strasbourg, France}}
\email{vilmos.komornik@gmail.com}
\author[P. Loreti]{Paola Loreti}
\address{Sapienza Universit\`a di Roma,
Dipartimento di Scienze di Base e Applicate per l'Ingegne\-ria,
via A. Scarpa n. 16,
00161 Roma, Italy}
\email{paola.loreti@sbai.uniroma1.it}
\date{Version of 2016-05-04-a}

\begin{abstract}
Cantor's ternary function is generalized to arbitrary base-change functions in non-integer bases. 
Some of them share the curious properties of Cantor's function, while others behave quite differently.
\end{abstract}
\maketitle


\section{Introduction}\label{s1}

Cantor's celebrated ternary function has been constructed by combining expansions of real numbers in two different bases $p=3$ and $q=2$.
This function has the surprising property to be a non-constant, continuous and non-decreasing function, having a zero derivative almost everywhere.
We refer to \cite{HilleTamarkin1929} for a survey of this and various other interesting features of Cantor's  function.

The purpose of this paper is to put Cantor's construction into a general framework by considering arbitrary bases $p,q>1$.
It turns out that for some values of $p$ and $q$ these functions have similar properties, while for other values they exhibit a quite different behavior.

Given a real \emph{base} $p>1$, by an \emph{expansion} of a real number $x$ in base $p$ we mean a sequence $c=(c_i)\in\set{0,1}^{\infty}$ satisfying the equality 
\begin{equation}\label{11}
\pi_p(c):=\sum_{i=1}^{\infty}\frac{c_i}{p^i}=x.
\end{equation}
We denote by $J_p$ the set of numbers $x$ having at least one expansion.
It is clear that $J_p\subseteq [0,\frac{1}{p-1}]$.

For $p=2$ the definition reduces to the familiar binary expansions of the numbers $x\in [0,1]$.
If $x\in (0,1)$ is a binary rational number, then it has two expansions: one ending with $0^{\infty}$, and another one ending with $1^{\infty}$.\footnote{We apply the notations of symbolic dynamics, i.e, $d^{\infty}$ denotes the constant sequence $d,d,\ldots ,$ and $10^{\infty}$ denotes the sequence $1,0,0,\ldots .$}
The remaining numbers in $[0,1]$ have a unique expansion.

The case $1<p<2$ was first investigated by R\'enyi \cite{Ren1957}.
He has proved among others the equality $J_p=[0,\frac{1}{p-1}]$  for each $p\in (1,2]$ by applying a greedy algorithm for each $x\in [0,\frac{1}{p-1}]$ as follows.
If $b_i$ has already been defined for all $i<n$ (no assumption if $n=1$), then we set $b_n=1$ if 
\begin{equation}\label{12}
\left(\sum_{i=1}^{n-1} \frac{b_i}{p^i}\right) +\frac{1}{p^n}\le x,
\end{equation}
and $b_n=0$ otherwise. 
Then $b_p(x):=(b_i)$ is an expansion of $x$ in base $p$.\footnote{Sometimes we write $b_i(x,p)$ instead of $b_i$  to show the dependence on $x$ and $p$.}
By construction this is the lexicographically largest expansion of $x$ in base $p$, called the \emph{greedy} or \emph{$\beta$-expansion} of $x$ in base $p$. 

Today there is a huge literature devoted to non-integer expansions.
For example, probabilistic and ergodic aspects are investigated in
\cite{DajDev2007}, \cite{DajKra2002}, \cite{Ren1957}, \cite{Sid2003}, \cite{SidVer1998},
combinatorial properties in 
\cite{BaiKom2007}, \cite{DajDevKomLor2012}, \cite{Gelfond1959}, \cite{KomLor2007b}, \cite{KomLor2010}, \cite{Par1960},
unique expansions in
\cite{DarKat1993}, \cite{DarKat1995}, \cite{Dev2008}, \cite{DevKom2009}, \cite{DevKom2011}, \cite{DevKomLor2016}, \cite{ErdHorJoo1991}, \cite{ErdJooKom1990}, \cite{ErdKom1998}, \cite{KomLor1998}, \cite{KomLor1999}, \cite{KomLor2002}, \cite{KomLor2007}, \cite{KomLorPet2003},
and control-theoretical applications are given in
\cite{ChiPic2001}, \cite{LaiLor2012}, \cite{LaiLor2015}.

Many other references are given in the surveys \cite{DevKom2016}, \cite{Kom2011}, \cite{Sid2010}.

The situation for $1<p<2$ is quite different from that of $p=2$.
For example, while in base $p=2$ each $x\in [0,1]$ has one or two expansions, in the bases $1<p<2$ almost every $x\in J_p$ has $2^{\aleph_0}$ expansions by a theorem of Sidorov \cite{Sid2003}.

Another difference is that the functions $x\mapsto b_p(x)$ of $J_p$ into $\set{0,1}^{\infty}$ are not monotone if $1<p<2$, and their behavior depends critically on the value of $p$.
In order to understand their mutual behavior we introduce for $p\in(1,2]$ and $q>1$ the \emph{base-change functions}
\begin{equation*}
b_{p,q}:=\pi_q\circ b_p:J_p\to J_q.
\end{equation*}
Explicitly, we have
\begin{equation}\label{13}
b_{p,q}(x):=\sum_{i=1}^{\infty} \frac{b_i}{q^i}\qtq{with}(b_i)=b_p(x).
\end{equation}
We exclude henceforth the trivial case $p=q$ when we get the identity function on $J_p$.

Now we may state our first result:

\begin{theorem}\label{t11}
Let $p\in(1,2]$ and $q>1$.

\begin{enumerate}[\upshape (i)]
\item The function $b_{p,q}$ is right continuous.
It is left continuous in $x\in J_p$ if and only if $b_p(x)$ contains infinitely many $1$ digits.
Its discontinuities form a countable dense set.
\item If $q<p$, then $b_{p,q}$ is nowhere monotone, not left differentiable anywhere, and not right differentiable in $x$ if the greedy expansion $b_p(x)$ contains at most finitely many non-zero digits.
\item If $q>p$, then $b_{p,q}$ is increasing, and hence differentiable almost everywhere.\end{enumerate}
\end{theorem}

Dar\'oczy and K\'atai \cite{DarKat1993}, \cite{DarKat1995} have introduced a slight variant of the $\beta$-expansions where the inequality ``$\le$'' in \eqref{12} is changed to the strict inequality ``$<$''.
More precisely, for $x=0$ we set $a_p(x):=0^{\infty}$. 
For $x\in(0,\frac{1}{p-1}]$ we define a sequence $a_p(x):=(a_i)$ as follows. 
If $a_i$ has already been defined for all $i<n$ (no assumption if $n=1$), then we set $a_n=1$ if 
\begin{equation*}
\left(\sum_{i=1}^{n-1} \frac{a_i}{p^i}\right) +\frac{1}{p^n}< x,
\end{equation*}
and $a_n:=0$ otherwise. 
Then $a_p(x)$ is again an expansion of $x$ in base $p$; it is called its \emph{quasi-greedy expansion}. 

In the present context it is customary to call a sequence \emph{finite} if it ends with $10^{\infty}$, and \emph{infinite} otherwise: hence the infinite sequences are $0^{\infty}$ and the sequences containing infinitely many $1$ digits.
Using this terminology, $a_p(x)$ is the lexicographically largest \emph{infinite} expansion of $x$ in base $p$.

Parry \cite{Par1960} gave a lexicographic characterization of greedy expansions by using the quasi-greedy expansion $a_p(1)$ of $x=1$, and  a similar characterization of quasi-greedy expansions was given in \cite{BaiKom2007}:

\begin{theorem*}
Let $p\in (1,2]$ and $(c_i)\in\set{0,1}^{\infty}$. 

\begin{enumerate}[\upshape (i)]
\item We have $(c_i)=b_p(x)$ for some $x\in [0,\frac{1}{p-1}]$ if and only if
\begin{equation*}
c_{n+1}c_{n+2}\cdots<a_p(1)\qtq{whenever} c_n=0.
\end{equation*} 
\item We have $(c_i)=a_p(1)$ for some $x\in [0,\frac{1}{p-1}]$ if and only if
\begin{equation*}
c_{n+1}c_{n+2}\cdots\le a_p(1)\qtq{whenever} c_n=0.
\end{equation*} 
\end{enumerate}
\end{theorem*}

\begin{remark}\label{r12}
We have $a_p(x)\ne b_p(x)$ if and only if $b_p(x)$ is finite. 
Hence the two expansions differ only for countably many values of $x$.
These values are dense in $J_p$.
Indeed, for any fixed $x\in J_p$ the truncated finite sequences
\begin{equation*}
b_1(x,p)\cdots b_n(x,p)0^{\infty},\quad n=1,2,\ldots
\end{equation*}
are greedy expansions by Parry's theorem, say 
\begin{equation*}
b(x_n,p)=b_1(x,p)\cdots b_n(x,p)0^{\infty},\quad n=1,2,\ldots,
\end{equation*}
and then $x_n\to x$.

Observe that the discontinuity points of $b_{p,q}$ are exactly the points $x$ where $a_p(x)\ne b_p(x)$.
\end{remark}

There is a variant of Theorem \ref{t11} for quasi-greedy expansions.
We introduce for $p\in(1,2]$ and $q>1$ the \emph{quasi-greedy base-change functions}
\begin{equation*}
a_{p,q}:=\pi_q\circ a_p:J_p\to J_q.
\end{equation*}
Explicitly, we have
\begin{equation*}
a_{p,q}(x):=\sum_{i=1}^{\infty} \frac{a_i}{q^i}\qtq{with}(a_i)=a_p(x).
\end{equation*}
We again exclude the trivial case $p=q$.

\begin{theorem}\label{t13}
Let $p\in(1,2]$ and $q>1$.

\begin{enumerate}[\upshape (i)]
\item The function $a_{p,q}$ is left continuous.
It is right continuous in $x\in [0,\frac{1}{p-1})$ if and only if $a_p(x)=b_p(x)$.
\item If $q<p$, then $a_{p,q}$ is nowhere monotone, not left differentiable anywhere, and not right differentiable in $x$ if the greedy expansion $b_p(x)$ contains at most finitely many non-zero digits.
\item If $q>p$, then $a_{p,q}$ is increasing, and hence differentiable almost everywhere.\end{enumerate}
\end{theorem}
The statements (ii), (iii) of Theorems \ref{t11} and \ref{t13} are similar, the proofs of the crucial properties (ii) are quite different: see the proof of Proposition \ref{p42} below.

Now we turn to the case $p>2$.
In this case $J_p$ is a proper subset of $[0,\frac{1}{p-1}]$, and each $x\in J_p$ has a unique expansion.
Indeed, if $(c_i), (d_i)$ are two different sequences with $c_1\cdots c_{n-1}=d_1\cdots d_{n-1}$ and $c_n>d_n$ for some $n\ge 1$, then
\begin{equation*}
\sum_{i=1}^{\infty}\left(\frac{c_i}{p^i}-\frac{d_i}{p^i}\right)
\ge \frac{1}{p^n}-\sum_{i=n+1}^{\infty}\frac{1}{p^i}
=\frac{p-2}{p^n(p-1)}>0.
\end{equation*}

We will clarify its topological nature of $J_p$:

\begin{theorem}\label{t14}
Let $p>2$.

\begin{enumerate}[\upshape (i)]
\item $J_p$ is a Cantor set  of Hausdorff dimension $\frac{1}{\log p}$, and hence a null set.
\item $x\in J_p$ is a right accumulation point of $J_p$ if and only if its unique expansion has infinitely many zero digits.
\item $x\in J_p$ is a left accumulation point of $J_p$ if and only if its unique expansion has infinitely many  one digits.
\end{enumerate}
\end{theorem}

Denoting by $b_p(x)=(b_i)$ the unique expansion of $x$ in base $p>2$, we may extend the definition $b_{p,q}:=\pi_q\circ b_p:J_p\to J_q$ for all $p,q>1$; explicitly,
\begin{equation*}
b_{p,q}(x):=\sum_{i=1}^{\infty} \frac{b_i}{q^i}\qtq{with} (b_i)=b_p(x).
\end{equation*}
However, the behavior of the functions $b_{p,q}$ is quite different for $p>2$:

\begin{theorem}\label{t15}
Let $p>2$ and $q>1$.

\begin{enumerate}[\upshape (i)]
\item The function $b_{p,q}$ is H\"older continuous with the exact exponent $\frac{\log q}{\log p}$.
\item If $q<p$, then $b_{p,q}$ is nowhere monotone.
\item If $q=2$, then $b_{p,q}$ is non-decreasing.
\item If $q>2$, then $b_{p,q}$ is increasing.
\end{enumerate}
\end{theorem}
Note that for $q>p>2$ the Hölder exponent in (i) is greater than one. 
This does not contradict the classical result that only constant functions have Hölder exponents greater than one, because the domain $J_p$ of $b_{p,q}$ is not an interval if $p>2$.

If $p>2$, then $J_p$ is a closed set in $[0,\frac{1}{p-1}]$, containing the endpoints of this interval.
Therefore we may extend the function $b_{p,q}:J_p\to J_q$ to a continuous function defined on  $[0,\frac{1}{p-1}]$.
Adapting the construction of Cantor's ternary function (it corresponds to the case $p=3$ and $q=2$) we introduce the continuous and surjective function
\begin{equation*}
B_{p,q}:\left[0,\frac{1}{p-1}\right]\to \left[0,\frac{1}{q-1}\right]
\end{equation*} 
that coincides with $b_{p,q}$ on $J_p$, and is affine on each connected component of $[0,\frac{1}{p-1}]\setminus J_p$.

\begin{theorem}\label{t16}
Let $p>2$ and $q>1$.

\begin{enumerate}[\upshape (i)]
\item The function $B_{p,q}$ is differentiable almost everywhere, and is H\"older continuous with the exact exponent $\min\set{1,\frac{\log q}{\log p}}$.
\item If $q<2$, then $B_{p,q}$ has no bounded variation.
\item If $q=2$, then $B_{p,q}$ is non-decreasing, but not absolutely continuous.
Its arc length is equal to $p/(p-1)$.
\item If $q>2$, then $B_{p,q}$ is increasing, absolutely continuous, and $B_{p,q}'>0$ a.e.
\end{enumerate}
\end{theorem}
The theorem shows that the above mentioned properties of Cantor's function remain valid for $q=2$ and all $p>2$, but not for $q\ne 2$.

The proof of the theorem (see equation \eqref{43} below) will also yield an explicit formula for the derivative.
If we denote by $I_{m,k}$, $k=1,\ldots, 2^m$ the removed intervals at the $m$th step of the construction of 
$B_{p,q}$ ($m=0,1,\ldots$), then
\begin{equation*}
B_{p,q}'(x)=\left(\frac{p}{q}\right)^{m+1}\frac{(q-2)(p-1)}{(p-2)(q-1)},\quad x\in I_{m,k}
\end{equation*}
for all $m$ and $k$.

The remainder of the paper is devoted to the proof of the above statements.
They may be easily adapted to the case of more general digits sets $\set{0,1,\ldots,M}$ with a given positive integer, by distinguishing the intervals $(1,M+1)$ and $(M+1,\infty)$ instead of $(1,2)$ and $(2,\infty)$.

\section{Study of the map $\pi_p$}
\label{s2}

The set $\set{0,1}^{\infty}$ of sequences $c=(c_i)$ is compact for the Tikhonov product topology, induced by 
the following metric: $\rho(c,d)=0$ if $c=d$, and $\rho(c,d)=2^{-n}$ if $c_1\cdots c_{n-1}=d_1\cdots d_{n-1}$ and $c_n\ne d_n$.
The corresponding convergence is the \emph{coordinate-wise convergence}.

The set $\set{0,1}^{\infty}$ also has a natural \emph{lexicographic order} and a corresponding \emph{order topology}.
The two topologies coincide:

\begin{proposition}\label{p21}
The product topology and the order topology coincide on $\set{0,1}^{\infty}$.
\end{proposition}

\begin{proof}
Each open ball is a finite intersection of open intervals, and hence open in the order topology:
\begin{equation*}
B_{2^{-n}}(c)=\cap_{j=1}^n A_j(c)
\end{equation*}
with
\begin{equation*}
A_j(c):=
\begin{cases}
\set{d\in \set{0,1}^{\infty}\ :\ d_1\cdots d_{j-1}>01^{\infty}}&\text{if $c_j=1$,}\\
\set{d\in \set{0,1}^{\infty}\ :\ d_1\cdots d_{j-1}<10^{\infty}}&\text{if $c_j=0$,}
\end{cases}
j=1,\ldots, n.
\end{equation*}

Conversely, each interval of the form
\begin{equation*}
A(c):=\set{d\in \set{0,1}^{\infty}\ :\ d>c}
\qtq{or}
\tilde A(c):=\set{d\in \set{0,1}^{\infty}\ :\ d<c}
\end{equation*}
(they form a subbase for the order topology) is open in the metric topology.
Indeed, for any fixed $d\in A(c)$ there exists an integer $n\ge 1$ such that $d_1\cdots d_n>c_1\cdots c_n$, and then $B_{2^{-n-1}}(d)\subset A(c)$ because 
\begin{equation*}
f\in B_{2^{-n-1}}(d) \Longrightarrow f_1\cdots f_n=d_1\cdots d_n>c_1\cdots c_n.
\end{equation*}

The proof for $\tilde A(c)$ is analogous.
\end{proof}

Henceforth we consider this metric and topology on $\set{0,1}^{\infty}$.

\begin{proposition}\label{p22}
The Hausdorff dimension of $\set{0,1}^{\infty}$ is equal to $1$.
\end{proposition}

\begin{proof}
The whole space is the compact invariant set of the iterated function system consisting of the two similarities of ratio $1/2$, defined by the formulas
\begin{equation*}
f_i(c_1c_2\cdots):=(ic_1c_2\cdots),\quad i=0,1.
\end{equation*}
Since their images are disjoint, Applying \cite[Theorem 6.4.3]{Edgar2008} (see also \cite{Falconer2003}) we obtain that the Hausdorff dimension is the solution of the equation 
\begin{equation*}
2^{-d}+2^{-d}=1,
\end{equation*}
i.e., $d=1$.
\end{proof}

Next we investigate the function $\pi_p:\set{0,1}^{\infty}\to J_p$ introduced in \eqref{11}. 
In the following proposition and in the sequel we use base two logarithm.

\begin{proposition}\label{p23} \mbox{}
Let $q>1$.

\begin{enumerate}[\upshape (i)]
\item $\pi_q$ is continuous, and even H\"older continuous with the exact exponent $\alpha=\log p$; more precisely,
\begin{equation}\label{21} 
\abs{\pi_q(c)-\pi_q(d)}\le\frac{q}{q-1}\rho(c,d)^{\log q}
\end{equation} 
for all sequences $c$ et $d$.
Hence its range $J_q$ is a non-empty compact set.
\item If $1<q<2$, then $J_q= \left[0,\frac{1}{q-1}\right]$, and $\pi_q$ is nowhere monotone.
\item If $q=2$, then $J_q= \left[0,1\right]$, and $\pi_q$ is non-decreasing.
\item If $q>2$, then $\pi_q$ is an increasing homeomorphism of $\set{0,1}^{\infty}$ onto $J_q$.
Moreover, we have a converse inequality to \eqref{21} :
\begin{equation}\label{22} 
\frac{q-2}{q-1}\rho(c,d)^{\log q}\le\abs{\pi_q(c)-\pi_q(d)}
\end{equation} 
for all sequences $c$ et $d$.

Hence $J_q$ has Hausdorff dimension $\frac{1}{\log q}<1$ and therefore it is a null set. 
\end{enumerate}
\end{proposition}

\begin{proof}
(i) We prove the H\"older continuity of $\pi_q$. If $\rho(c,d)=2^{-n}$, then
\begin{equation*}
c_1\cdots c_{n-1}=d_1\cdots d_{n-1}
\end{equation*}
and therefore 
\begin{equation*}
\abs{\pi_q(c)-\pi_q(d)}
\le\sum_{i=n}^{\infty}\frac{1}{q^i}
=\frac{1}{q^{n-1}(q-1)}
=\frac{q}{q-1}\rho(c,d)^{\log q}.
\end{equation*}
The exponent $\alpha=\log q$ cannot be improved because for $c=1^{\infty}$ and $d=0^{\infty}$ we have equality.

Since $\set{0,1}^{\infty}$ is compact and $\pi_q$ is continuous, the range $J_q$ of $\pi_q$ is also compact.
\medskip

(ii) We already know that $J_q= \left[ 0,\frac{1}{q-1}\right]$. 

Given an arbitrary non-degenerate subinterval $I$ of $\left[0,\frac{1}{q-1}\right]$, there exist a large integer $n$ and a block $a_1\cdots a_{n-1}\in\set{0,1}^{n-1}$ such that $\pi_q(c)\in I$ for all sequences $c$ starting with $a_1\cdots a_{n-1}$.
Consider the sequences 
\begin{equation*}
b:=a_1\cdots a_{n-1}0^{\infty},\quad c:=a_1\cdots a_{n-1}01^{\infty}\qtq{and} d:=a_1\cdots a_{n-1}10^{\infty}.
\end{equation*}
We have
\begin{equation*}
b<c<d
\qtq{and} 
\pi_q(b),\ \pi_q(c),\ \pi_q(d)\in I.
\end{equation*}
Furthermore, we have obviously $\pi_q(b)<\pi_q(c)$, and also $\pi_q(c)>\pi_q(d)$ because 
\begin{equation*}
\pi_q(d)-\pi_q(c)=\frac{1}{q^n}-\sum_{i=n+1}^{\infty}\frac{1}{q^i}=\frac{(q-2)}{q^n(q-1)}<0.
\end{equation*}
Hence $\pi_q$ is not monotone in $I$.
\medskip

(iii) The property $J_2=[0,1]$ is well known. 

If two sequences $c<d$ first differ at their $n$th digits, then 
\begin{equation*}
\pi_2(d)-\pi_2(c)\ge \frac{1}{2^n}-\sum_{i=n+1}^{\infty}\frac{1}{2^i}=0.
\end{equation*}
Hence $\pi_2$ is non-decreasing.

It is not (strictly) increasing because $\pi_2(10^{\infty})=\pi_2(01^{\infty})$. 
\medskip

(iv) If $p>2$, and two sequences $c<d$ first differ at their $n$th digits, then 
\begin{equation*}
\pi_q(d)-\pi_q(c)\ge \frac{1}{q^n}-\sum_{i=n+1}^{\infty}\frac{1}{q^i}=\frac{(q-2)}{q^n(q-1)}>0.
\end{equation*}
This proves the increasingness of $\pi_q$ and the  inequality  \eqref{22} . 

It follows from \eqref{21}  and \eqref{22}  that $\pi_q$ is a homeomorphism.

Finally, since $\set{0,1}^{\infty}$ has Hausdorff dimension $1$, using \eqref{21}  and \eqref{22}  we infer from the definition of the that Hausdorff dimension $J_q$ has Hausdorff dimension  $1/\log q<1$.
\end{proof}

Next we investigate the greedy and quasi-greedy maps $b_p:J_p\to\set{0,1}^{\infty}$ and $a_p:J_p\to\set{0,1}^{\infty}$, defined in the introduction.
We have 
\begin{equation*}
\pi_p(b_p(x))=\pi_p(a_p(x))=x\qtq{for all} x\in J_p
\end{equation*} 
by definition, i.e., $\pi_p$ is a left inverse of both $b_p$ and $a_p$.

\begin{proposition}\label{p24} 
We have 
\begin{equation*}
1<p<q\le 2\Longrightarrow R(b_p)\subsetneq R(b_q)\qtq{and} R(a_p)\subsetneq R(a_q),
\end{equation*}
and 
\begin{equation*}
R(b_p)=\set{0,1}^{\infty}\qtq{for all} p>2.
\end{equation*}
\end{proposition}

\begin{proof}
The inclusions $\subseteq$ and for $p>2$ the equalities follow from the definitions of the greedy and quasi-greedy expansions: 
\begin{align*}
(c_i)\in R(b_p) &\Longleftrightarrow \sum_{i=1}^{\infty}\frac{c_{n+i}}{p^i}<1\qtq{whenever} c_n=0,
\intertext{and} 
(c_i)\in R(a_p) &\Longleftrightarrow \sum_{i=1}^{\infty}\frac{c_{n+i}}{p^i}\le 1\qtq{whenever} c_n=0,
\end{align*}
because both inequalities remain valid if we change $p$ to a larger $q$.

It remains to show that for $1<p<q\le 2$ the sets $R(b_q)\setminus R(b_p)$ and $R(a_q)\setminus R(a_p)$ are non-empty. 
It follows from \cite[Theorem 2.7]{DevKom2011} that $R(b_q)\setminus R(b_p)$ has in fact $2^{\aleph_0}$ elements. 
Since the sets $R(b_q)\setminus R(a_q)$ and $R(a_p)\setminus R(b_p)$ are countable by Remark \ref{r12}, $R(a_q)\setminus R(a_p)$ has also $2^{\aleph_0}$ elements. 
\end{proof}

We recall that $J_p= \left[ 0,\frac{1}{p-1}\right]$ is an interval if $p\in (1,2]$.
Now we clarify the topological picture of $J_p$ for $p>2$.
We recall that each $x\in J_p$ has a unique expansion. 

\begin{proposition}\label{p25} \mbox{}
Let $p>2$ and $x\in J_p$.

\begin{enumerate}[\upshape (i)]
\item $x\in J_p$ is a right accumulation point of $J_p$ if and only if its unique expansion has infinitely many zero digits.
\item $x\in J_p$ is a left accumulation point of $J_p$ if and only if its unique expansion has infinitely many  one digits.
\item $J_p$ is a Cantor set, i.e., a non-empty compact set having neither interior, nor isolated points.
\end{enumerate}
\end{proposition}

\begin{proof}
(i) If the unique expansion $(c_i)$ of $x$ has infinitely many zero digits $c_n=0$, then the formula 
\begin{equation*}
x_n:=\left( \sum_{i=0}^{n-1}\frac{c_i}{p^i}\right)+\frac{1}{p^n},
\end{equation*}
where $n$ runs over the integers for which $c_n=0$, defines a sequence $(x_n)$ converging to $x$ and satisfying $x_n>x$ for all $n$. 
Hence $x$ is  a right accumulation point.

Otherwise either $x=\max J_p$ or the unique expansion $(c_i)$ of $x$ has a last zero digit $c_m=0$. 
In the first case $x$ is obviously a right isolated point. 
In the second case the unique expansion of each $y\in J_p$, $y>x$ starts with some block $d_1\cdots d_m>c_1\cdots c_m$, so that 
\begin{equation*}
y-x\ge \frac{1}{p^m}-\sum_{i=m+1}^\infty\frac{1}{p^i}=\frac{p-2}{p^m(p-1)}>0.
\end{equation*}
Since the right side does not depend on the particular choice of $y$, we conclude that $x$ is  a right isolated point again.
\medskip

(ii) The statements (i) and (ii) are equivalent because $J_p$ is symmetric with respect to $\frac{1}{2(p-1)}$, and the expansion  of any $x\in J_p$ is the \emph{reflection} $(1-c_i)$ of the expansion $(c_i)$ of $\frac{1}{p-1}-x$.
\medskip

(iii) We already know that $J_p$ is a non-empty compact set.
We also know that it is a null set, hence $J_p$ has no interior points. 
Finally, since every sequence has infinitely many equal digits, all points of $J_p$ are accumulation points by (i) and (ii).
\end{proof}

We end this section by investigating the continuity of the maps $b_p$ and $a_p$.

\begin{proposition}\label{p26} \mbox{}
Let $p>1$ and $x\in J_p$.

\begin{enumerate}[\upshape (i)]
\item If $p\in (1,2]$, then $b_p$ is right continuous and $a_p$ is left continuous.
Furthermore, they are continuous in $x$ if and only if $b_p(x)=a_p(x)$. 
\item If $p>2$, then $b_p$ is a homeomorphism between $J_p$ and $\set{0,1}^{\infty}$. 
Moreover,
\begin{equation*}
c_1\abs{x-y}^{1/\log p} 
\le \rho(b_p(x),b_p(y))
\le c_2\abs{x-y}^{1/\log p}
\end{equation*}
for all $x,y\in J_p$ with 
\begin{equation*}
c_1:=\left(\frac{p-1}{p}\right)^{1/\log p} 
\qtq{and} 
c_2:=\left(\frac{p-1}{p-2}\right)^{1/\log p}.
\end{equation*}
\end{enumerate}
\end{proposition}

\begin{proof}
(i) Let $p\in (1,2]$, and consider a sequence $(x_k)$ converging to $x$ in $J_p$. 

If $x_k>x$ for all $k$, then $b_p(x_k)\to b_p(x)$ coordinate-wise: this follows from the definition of the greedy algorithm (or see \cite[Lemma 2.5]{DevKom2011}). 
This shows that $b_p$ is right continuous.

Furthermore, since 
\begin{equation*}
b_p(x_k)\ge a_p(x_k)>b_p(x)\ge a_p(x),
\end{equation*}
we have also $a_p(x_k)\to b_p(x)$ coordinate-wise. 
Therefore $a_p$ is right continuous in $x$ if and only if $b_p(x)=a_p(x)$.

If $x_k<x$ for all $k$, then $a_p(x_k)\to a_p(x)$ coordinate-wise: this follows from the definition of the quasi-greedy algorithm (or see \cite[Lemma 2.3]{DevKom2011}). 
This shows that $a_p$ is left continuous.

Furthermore, since 
\begin{equation*}
a_p(x_k)\le b_p(x_k)<a_p(x)\le b_p(x),
\end{equation*}
it follows that $b_p(x_k)\to a_p(x)$ coordinate-wise. 
Therefore $b_p$ is left continuous in $x$ if and only if $b_p(x)=a_p(x)$.
\medskip

(ii) This follows from Proposition \ref{p23}  (i) and (iv) because $b_p$ is the inverse of the homeomorphism $\pi_p$.
\end{proof}

\begin{remarks}\label{r27} \mbox{}

\begin{enumerate}[\upshape (i)]
\item We may also give a direct proof of the discontinuity of $b_p$ and $a_p$ at $x$ when $b_p(x)=(b_i)$ has a last non-zero digit $b_m=1$. 
We show that $\rho(b_p(y),b_p(x))\ge 2^{-m}>0$ for all $y<x$, and $\rho(b_p(y),a_p(x))\ge 2^{-m}>0$ for all $y>x$.

Indeed, for $y<x$ we have $b_p(y)<b_1\cdots b_m0^{\infty}$ and therefore
\begin{equation*}
\rho(b_p(y),b_p(x))\ge 2^{-m}>0.
\end{equation*} 

Similarly, for $y>x$ we have $a_p(y)> b_1\cdots b_m0^{\infty}$.
Since $a_p(x)=(b_1\cdots b_{m-1}0)^{\infty}$,  this implies that $\rho(a_p(y),a_p(x))\ge 2^{-m}>0$.

\item By Remark \ref{r12} the discontinuities of $b_p$ and $a_p$ form a countable dense set in $J_p$.

\item If we endow $R(b_p)$ and $R(a_p)$ with the topology associated with the lexicographic order on these sets, then $b_p$ and $a_p$ remain homeomorphisms. 
This shows that this topology is different from the topology associated with the restriction of the metric $\rho$: the latter one is finer.

Indeed, the second part of the proof of Proposition \ref{p21} remains valid in each subset of $\set{0,1}^{\infty}$, but the first may fail.
Consider for example the open balls
\begin{align*}
&B_{1/2}(0^{\infty})=\set{c\in \set{0,1}^{\infty}\ :\ c<10^{\infty}}
\intertext{and}
&B_{1/2}(1^{\infty})=\set{c\in \set{0,1}^{\infty}\ :\ c>01^{\infty}}
\end{align*}
in $\set{0,1}^{\infty}$.

Then $B_{1/2}(0^{\infty})\cap R(a_p)$ is not open in the order topology of $R(a_p)$ because it has a maximal element $01^{\infty}$, and each order-neighborhood of $01^{\infty}$ contains larger sequences in $R(a_p)$.

Similarly, $B_{1/2}(1^{\infty})\cap R(b_p)$ is not open in the order topology of $R(b_p)$ because it has a minimal element $10^{\infty}$, and each order-neighborhood of $10^{\infty}$ contains smaller sequences in $R(b_p)$.
\end{enumerate}
\end{remarks}

\section{Monotonicity and continuity of base-change functions}\label{s3} 

In this section we investigate the \emph{base-change functions}
\begin{align*}
&b_{p,q}:=\pi_q\circ b_p:J_p\to J_q \qtq{for}p,q>1,
\intertext{and}
&a_{p,q}:=\pi_q\circ a_p:J_p\to J_q \qtq{for}p\in (1,2]\qtq{and}q>1.
\end{align*}
We exclude the trivial case $p=q$ where they are the identity maps of $J_p$.

First we study $b_{p,q}$. 

\begin{proposition}\label{p31} \mbox{}
Let $p,q>1$ with $p\ne q$. 

\begin{enumerate}[\upshape (i)]
\item If $q>\min\set{p,2}$, then $b_{p,q}$ is increasing.
\item If $q=2<p$, then  $b_{p,q}$ is non-decreasing.
\item In the remaining case $q<\min\set{p,2}$ the map $b_{p,q}$ is nowhere monotone.
\end{enumerate}
\end{proposition}

\begin{proof}
(i) If $q>2$, then both $b_p$ and $\pi_q$ are increasing, hence $b_{p,q}=\pi_q\circ b_p$ is also increasing.

If $q>p$, then both $b_p$ and the restriction of $\pi_q$ to $R(b_p)\subseteq R(b_q)$ are increasing, hence $b_{p,q}=\pi_q\circ b_p$ is also increasing.
\medskip

(ii) If $q=2<p$, then $b_p$ is increasing and $\pi_q$ is non-decreasing, so that $b_{p,q}=\pi_q\circ b_p$ is non-decreasing. It is not increasing, however, because for example $\pi_2(10^{\infty})=\pi_2(01^{\infty})$, whence $b_{p,2}$ takes the same value at the points
$b_p^{-1}(01^{\infty})<b_p^{-1}(10^{\infty})$.
\medskip

(iii) Assume henceforth that $q<\min\set{p,2}$, and fix an arbitrary open interval $I\subseteq [0,\frac{1}{p-1}]$ such that $I\cap J_p\ne\varnothing$.

Fix a point $x\in I\cap J_p$ whose greedy expansion $b_p(x)=(b_i)$ is infinite, and choose a large integer $n$ such that $b_n=1$, and $\pi_p(c)\in I$ for all $p$-greedy sequences $c$ starting with $b_1\cdots b_{n-1}$.

Furthermore, fix a number $r\in (q,\min\set{p,2})$ and consider the quasi-greedy expansion $(\alpha_i):=a_r(1)$. (It is well defined because $r\in(1,2)$.)

Since $r<p$, 
\begin{equation*}
b_1\cdots b_{n-1}0^{\infty}<b_1\cdots b_{n-1}0\alpha_1\alpha_2\cdots<b_1\cdots b_{n-1}10^{\infty}
\end{equation*}
are greedy expansions in base $p$ of suitable numbers $x<y<z$ (we use here the increasingness of the map $x\mapsto b_p(x)$).
Since $x,y,z\in I$ by the choice of $n$, the proof will be completed by showing that 
$\pi_q(y)>\pi_q(x)$ and $\pi_q(y)>\pi_q(z)$.

The first relation is obvious: 
\begin{equation*}
\pi_q(y)-\pi_q(x)=\sum_{i=1}^{\infty}\frac{\alpha_i}{q^{n+i}}>0.
\end{equation*}
The second relation follows from our assumption $q<r$: 
\begin{equation*}
q^n\left(\pi_q(y)-\pi_q(z)\right)=\left(\sum_{i=1}^{\infty}\frac{\alpha_i}{q^i}\right)-1 >\left(\sum_{i=1}^{\infty}\frac{\alpha_i}{r^i}\right)-1 =0.\qedhere
\end{equation*}
\end{proof}

\begin{corollary}\label{c32} \mbox{}
Let $p,q>1$ with $p\ne q$. 

\begin{enumerate}[\upshape (i)]
\item $b_{p,q}$ is injective if and only if $q>\min\set{p,2}$.
\item $b_{p,q}$ is onto if and only if $p>\min\set{q,2}$.
\item $b_{p,q}$ is bijective if and only if $\min\set{p,q}>2$.
\end{enumerate}
\end{corollary}

\begin{proof}
Part (i) follows from the preceding proposition. 

Since $\pi_q$ is a bijection between $R(b_q)$ and $J_q$, $b_{p,q}=\pi_q\circ b_p$ is onto if and only if $R(b_q)\subseteq R(b_p)$. 
In view of Proposition \ref{p24}  this is equivalent to the condition $p>\min\set{q,2}$.
This proves (ii).

Parts (i) and (ii) imply (iii).
\end{proof}

Next we investigate the continuity:

\begin{proposition}\label{p33} \mbox{}
Let $p,q>1$ with $p\ne q$. 

\begin{enumerate}[\upshape (i)]
\item If $p,q>2$, then $b_{p,q}$ is a homeomorphism.
\item If $p>2$, then $b_{p,q}$ is Hölder continuous with the exact exponent $\frac{\log q}{\log p}$. 
\item If $p\in(1,2]$, then $b_{p,q}$ is right continuous.
It is left continuous in $x\in (0,\frac{1}{p-1}]$ if and only if $b_p(x)$ has infinitely many $1$ digits.
\end{enumerate}
\end{proposition}

\begin{remark}\label{r34} 
For $q>p>2$ the Hölder exponent is greater than one. 
This is possible because $b_p$ is not defined on an interval.
\end{remark}

\begin{proof}[Proof of Proposition \ref{p33}]
Propositions \ref{p23} and \ref{p26} imply most statements by  the continuity of composite functions.
It remains to prove that if  $p\in (1,2]$ and $(b_i):=b_p(x)$ has a last non-zero digit $b_n=1$, then $b_{p,q}$ is not left continuous in $x$. 

Setting $(\alpha_i):=a_p(1)$ for brevity we have
\begin{equation*}
b_p(x)=b_1\cdots b_{n-1}b_n0^{\infty}\qtq{and}a_p(x)=b_1\cdots b_{n-1}(b_n-1)\alpha_1\alpha_2\cdots.
\end{equation*}
Furthermore, if $y<x$ and $y\to x$, then  $b_p(y)\to a_p(x)$ (see the proof of Proposition \ref{p26} (i)), 
and therefore
\begin{align*}
b_{p,q}(x)-\lim_{y\nearrow x}b_{p,q}(y)
&=\frac{1}{q^n}\left(1-\sum_{i=1}^{\infty}\frac{\alpha_i}{q^i}\right)\\
&\ne \frac{1}{q^n}\left(1-\sum_{i=1}^{\infty}\frac{\alpha_i}{p^i}\right)\\
&=0.\qedhere
\end{align*}
\end{proof}

\begin{remark}\label{r35}
The last proof shows that if $1<q<p\le 2$ and $b_p(x)$ has a last nonzero digit, then 
\begin{equation}\label{31}
\lim_{y\nearrow x}b_{p,q}(y)>b_{p,q}(x).
\end{equation}
\end{remark}

Next we consider the same questions for the functions $a_{p,q}$.
Now we have $p\in (1,2]$ by definition.

\begin{proposition}\label{p36} \mbox{}
Let $p\in (1,2]$ and $q>1$ with $p\ne q$. 

\begin{enumerate}[\upshape (i)]
\item If $q>p$, then $a_{p,q}$ is increasing.
\item If $q<p$ then $a_{p,q}$ is nowhere monotone.
\end{enumerate}
\end{proposition}

\begin{proof}
(i) If $q>p$, then both $a_p$ and the restriction of $\pi_q$ to $R(a_p)\subseteq R(a_q)$ are increasing, hence $a_{p,q}=\pi_q\circ a_p$ is also increasing.
\medskip

(ii) Let $q<p$, and fix an arbitrary open interval $I\subseteq [0,\frac{1}{p-1}]$ such that $I\cap J_p\ne\varnothing$.

Fix an arbitrary point $x\in I\cap J_p$, write $(a_i)=a_p(x)$ for brevity, and choose a large integer $n$ such that $a_n=1$, and $\pi_p(c)\in I$ for all $p$-quasi-greedy sequences $c$ starting with $a_1\cdots a_{n-1}0$.

Next write $(\alpha_i)=a_p(1)$ for brevity, and choose a large integer $m$ such that 
\begin{equation*}
10^m\alpha_1\alpha_2\cdots < \alpha_1\alpha_2\cdots
\qtq{and} 
\left(1-\frac{1}{q^m}\right)\sum_{i=1}^{\infty}\frac{\alpha_i}{q^i}>1.
\end{equation*}
This is possible because $\alpha_1=1$, $\alpha_1\alpha_2\cdots>10^{\infty}$, and
\begin{equation*}
\sum_{i=1}^{\infty}\frac{\alpha_i}{q^i}>\sum_{i=1}^{\infty}\frac{\alpha_i}{p^i}=1.
\end{equation*}

It follows from the lexicographic characterization of quasi-greedy expansions and from our first assumption on $m$ that 
\begin{equation*}
a_1\cdots a_{n-1}000\alpha_1\alpha_2\cdots<a_1\cdots a_{n-1}00\alpha_1\alpha_2\cdots<a_1\cdots a_{n-1}010^m\alpha_1\alpha_2\cdots
\end{equation*}
are quasi-greedy expansions in base $p$ of suitable numbers $x<y<z$ (we use here the increasingness of the map $x\mapsto a_p(x)$).
Since $x,y,z\in I$ by the choice of $n$, the proof will be completed by showing that 
$\pi_q(y)>\pi_q(x)$ and $\pi_q(y)>\pi_q(z)$.

The first relation is obvious: 
\begin{equation*}
\pi_q(y)-\pi_q(x)=\left(\frac{1}{q^{n+1}}-\frac{1}{q^{n+2}}\right)\sum_{i=1}^{\infty}\frac{\alpha_i}{q^i}>0.
\end{equation*}
The second relation follows from our second assumption on $m$:
\begin{equation*}
q^{n+1}\left(\pi_q(y)-\pi_q(z)\right)=\left(\left(1-\frac{1}{q^m}\right)\sum_{i=1}^{\infty}\frac{\alpha_i}{q^i}\right)-1 >0.\qedhere
\end{equation*}
\end{proof}

\begin{corollary}\label{c37} \mbox{}
Let $p\in (1,2]$ and $q>1$ with $p\ne q$. 

\begin{enumerate}[\upshape (i)]
\item $a_{p,q}$ is injective if and only if $q>p$.
\item $a_{p,q}$ is onto if and only if $p>q$.
\item $a_{p,q}$ is never bijective.
\end{enumerate}
\end{corollary}

\begin{proof}
Part (i) readily follows from the preceding proposition. 

Since $\pi_q$ is a bijection between $R(a_q)$ and $J_q$, $a_{p,q}=\pi_q\circ a_p$ is onto if and only if $R(a_q)\subseteq R(a_p)$. 
In view of Proposition \ref{p24}  this is equivalent to the condition $p>q$.

Parts (i) and (ii) imply (iii).
\end{proof}

\begin{proposition}\label{p38} \mbox{}
Let $p\in (1,2]$ and $q>1$ with $p\ne q$. 

The function $a_{p,q}$ is left continuous.
It is right continuous in $x\in [0,\frac{1}{p-1})$ if and only if $b_p(x)=a_p(x)$.
\end{proposition}

\begin{proof}
Using the continuity of composite functions, Propositions \ref{p23} and \ref{p26} imply the positive continuity statements.
It remains to prove that if $(b_i):=b_p(x)$ has a last non-zero digit $b_n=1$, then $a_{p,q}$ is not right continuous in $x$. 

Setting $(\alpha_i):=a_p(1)$ for brevity we have
\begin{equation*}
b_p(x)=b_1\cdots b_{n-1}b_n0^{\infty}\qtq{and} a_p(x)=b_1\cdots b_{n-1}(b_n-1)\alpha_1\alpha_2\cdots.
\end{equation*}
Furthermore, if $y>x$ and $y\to x$, then $a_p(y)\to b_p(x)$ (see the proof of Proposition \ref{p26} (i)), 
and therefore
\begin{align*}
\lim_{y\searrow x}a_{p,q}(y)-a_{p,q}(x)
&=\frac{1}{q^n}\left(1-\sum_{i=1}^{\infty}\frac{\alpha_i}{q^i}\right)\\
&\ne \frac{1}{q^n}\left(1-\sum_{i=1}^{\infty}\frac{\alpha_i}{p^i}\right)\\
&=0.\qedhere
\end{align*}
\end{proof}

\begin{remark}\label{r39}
The last proof shows that if $1<q<p\le 2$ and $b_p(x)$ has a last nonzero digit, then 
\begin{equation}\label{32}
\lim_{y\searrow x}a_{p,q}(y)<a_{p,q}(x).
\end{equation}
\end{remark}

\section{Differentiability and bounded variation property}\label{s4} 

For the proofs of this section we recall from \cite[Lemmas 3.1 and 3.2]{DevKom2009} a property of greedy expansions:

\begin{lemma}\label{l41}\mbox{}
Let $p\in (1,2]$.

\begin{enumerate}[\upshape (i)]
\item If $(c_i)=b_p(x)$ or $(c_i)=a_p(x)$ for some $x\in J_p$, then for each $n\ge 1$ there exists $x_n\in J_p$ such that $x_n\le x$ and $b_p(x_n)=c_1\cdots c_n0^{\infty}$.
\item If $b_p(x)\ne 1^{\infty}$ for some $x\in J_p$, then for each $n\ge 1$ there exists $x_n\in J_p$ such that $x_n>x$ and 
\begin{equation*}
b_1(x_n,p)\cdots b_n(x_n,p)=b_1(x,p)\cdots b_n(x,p).
\end{equation*}
\end{enumerate}
\end{lemma}

In the following proposition we use the Dini derivatives of a function, defined by the formulas
\begin{equation*}
d_-f(x):=\liminf_{y\nearrow x}\frac{f(x)-f(y)}{x-y},\quad D_-f(x):=\limsup_{y\nearrow x}\frac{f(x)-f(y)}{x-y}
\end{equation*}
and
\begin{equation*}
d_+f(x):=\liminf_{y\searrow x}\frac{f(x)-f(y)}{x-y},\quad D_+f(x):=\limsup_{y\searrow x}\frac{f(x)-f(y)}{x-y}.
\end{equation*}
We recall that $f$ is differentiable in $x$ if and only if all four Dini derivatives exist, are finite and are equal in $a$.

\begin{proposition}\label{p42} \mbox{}

\begin{enumerate}[\upshape (i)]
\item If $1<q<p\le 2$, then $b_{p,q}$ and $a_{p,q}$ are not differentiable anywhere.
More precisely, their Dini derivatives satisfy the following relations:
\begin{enumerate}[\upshape (a)]
\item If $x=0$, then $d_+b_{p,q}(x)=d_+a_{p,q}(x)=\infty$.
\item If $x=1/(p-1)$, then $d_-b_{p,q}(x)=d_-a_{p,q}(x)=\infty$.
\item If $x\in (0,\frac{1}{p-1})$ and $b_p(x)=a_p(x)$, then $D_-b_{p,q}(x)=D_-a_{p,q}(x)=\infty$.
\item If $b_p(x)\ne a_p(x)$, then 
\begin{equation*}
D_-b_{p,q}(x)=D_+a_{p,q}(x)=-\infty
\qtq{and} 
D_+b_{p,q}(x)=D_-a_{p,q}(x)=\infty.
\end{equation*}
\end{enumerate}

\item If $p\in (1,2]$ and $q>p$, then $b_{p,q}$ and $a_{p,q}$ are differentiable almost everywhere. 
\item If $p>2$, then $B_{p,q}$ is differentiable with $B_{p,q}'>0$ almost everywhere. 
\end{enumerate}
\end{proposition}

\begin{proof}
(i) Assume that $1<q<p\le 2$. 
First we investigate the left differentiability of $b_{p,q}$ in a point $x\in (0,\frac{1}{p-1}]$.

If $x=1/(p-1)$, then $b_p(x)=1^{\infty}$. 
If $x_n\nearrow x$, then $(b_{n,i}):=b_p(x_n)$ starts with $1^m0$ where $m=m(n)\to\infty$ as $n\to\infty$. 
Therefore, writing $\overline{b_i}:=1-b_i$,
\begin{align*}
\frac{b_{p,q}(x)-b_{p,q}(x_n)}{x-x_n} 
&=\frac{\pi_q(0^m1\overline{b_{n,m+2}b_{n,m+3}\cdots})}{\pi_p(0^m1\overline{b_{n,m+2}b_{n,m+3}\cdots})}\\
&\ge \frac{\pi_q(0^m10^{\infty})}{\pi_p(0^m1 ^{\infty})}\\ 
&=\frac{p-1}{p}\left(\frac{p}{q}\right)^{m+1}
\end{align*}
for each $n$.
Since $p>q$, letting $n\to\infty$ this yields $d_-b_{p,q}(x)=\infty$.

If $x>0$ and $b_p(x)=a_p(x)$, then $(b_i):=b_p(x)$ has infinitely non-zero digits $b_n=1$.
For each such $n$, applying Lemma \ref{l41} there exists $x_n<x$ such that $b_p(x_n)=b_1\cdots b_{n-1}0^{\infty}$.
Then $x_n\to x$, and 
\begin{align*}
\frac{b_{p,q}(x)-b_{p,q}(x_n)}{x-x_n} 
&=\frac{\pi_q(b_1b_2\cdots)-\pi_q(b_1\cdots b_{n-1}0^{\infty})}{\pi_p(b_1b_2\cdots)-\pi_p(b_1\cdots b_{n-1}0^{\infty})} \\
&=\frac{\pi_q(0^{n-1}1b_{n+1}\cdots)}{\pi_p(0^{n-1}1b_{n+1}\cdots)}\\
&\ge \frac{\pi_q(0^{n-1}10^{\infty})}{\pi_p(0^{n-1}1 ^{\infty})}\\ 
&=\frac{p-1}{p}\left(\frac{p}{q}\right)^n
\end{align*}
for each $n$.
Letting $n\to\infty$ we conclude that  $D_-b_{p,q}(x)=\infty$.

If $b_p(x)\ne a_p(x)$, then using \eqref{31} we obtain that
\begin{equation*}
D_-b_{p,q}(x)=\lim_{y\nearrow x}\frac{b_{p,q}(x)-b_{p,q}(y)}{x-y}=-\infty.
\end{equation*}

Next we investigate the right differentiability of $b_{p,q}$ in a point $x\in [0,\frac{1}{p-1})$.

If $x=0$, then $b_p(x)=0^{\infty}$. 
If $x_n\searrow x$, then $(b_{n,i}):=b_p(x_n)$ starts with $0^m1$ where $m=m(n)\to\infty$ as $n\to\infty$. 
Therefore
\begin{align*}
\frac{b_{p,q}(x_n)-b_{p,q}(x)}{x_n-x} 
&=\frac{\pi_q(0^m1b_{n,m+2}b_{n,m+3}\cdots)}{\pi_p(0^m1b_{n,m+2}b_{n,m+3}\cdots)}
\ge \left(\frac{p}{q}\right)^m\frac{\pi_q(10^{\infty})}{\pi_p(1^{\infty})}
\end{align*}
for each $n$.
Letting $n\to\infty$ we conclude that
$d_+b_{p,q}(0)=\infty$.

If $b_p(x)\ne a_p(x)$, then $(b_i):=b_p(x)$ has a last non-zero element $b_n=1$.
By Lemma \ref{l41} there exist arbitrarily large integers $m>n$ such that
\begin{equation*}
b_1\cdots b_n0^{m-n}10^{\infty}=b_p(x_m)
\end{equation*}
for some $x_m>x$.
Then 
\begin{equation*}
\frac{b_{p,q}(x_m)-b_{p,q}(x)}{x_m-x} 
=\frac{\pi_q(0^m10^{\infty})}{\pi_p(0^m10^{\infty})}
=\left(\frac{p}{q}\right)^m.
\end{equation*}
Letting $m\to\infty$ we conclude that $D_+b_{p,q}(x)=\infty$.
\medskip 

Now we turn to the function $a_{p,q}$.
First we investigate its left differentiability in a point $x\in (0,\frac{1}{p-1}]$.

If $x=1/(p-1)$, then $a_p(x)=1^{\infty}$. 
If $x_n\nearrow x$, then $(a_{n,i}):=a_p(x_n)$ starts with $1^m0$ where $m=m(n)\to\infty$ as $n\to\infty$. 
Therefore, writing $\overline{a_i}:=1-a_i$,
\begin{align*}
\frac{a_{p,q}(x)-a_{p,q}(x_n)}{x-x_n} 
&=\frac{\pi_q(0^m1\overline{a_{n,m+2}a_{n,m+3}\cdots})}{\pi_p(0^m1\overline{a_{n,m+2}a_{n,m+3}\cdots})}\\
&\ge \frac{\pi_q(0^m10^{\infty})}{\pi_p(0^m1 ^{\infty})}\\ 
&=\frac{p-1}{p}\left(\frac{p}{q}\right)^{m+1}
\end{align*}
for each $n$.
Since $p>q$, letting $n\to\infty$ this yields $d_-a_{p,q}(x)=\infty$.

The following consideration is valid for all $x\in (0,\frac{1}{p-1}]$.
Set $(\alpha_i):=a_p(1)$.
Furthermore, fix a number $r\in (1,q)$ and set $(\tilde\alpha_i):=a_r(1)$.
Since $r<p$, we have $(\tilde\alpha_i)<(\alpha_i)$.

Since $x>0$, $(a_i):=a_p(x)$ has infinitely many nonzero digits $a_n=1$.
For each such $n$, $a_1\cdots a_n0^{\infty}$ is $p$-greedy, hence $a_1\cdots a_{n-1}0\alpha_1\alpha_2\cdots$ is $p$-quasi-greedy.

We claim that $(c_i):=a_1\cdots a_{n-1}0\tilde\alpha_1\tilde\alpha_2\cdots$ is also $p$-quasi-greedy, and hence
\begin{equation*}
a_p(x_n)=a_1\cdots a_{n-1}0\tilde\alpha_1\tilde\alpha_2\cdots
\end{equation*}
for some $x_n<x$.

For this we have to show that 
\begin{equation*}
(c_{k+i})\le (\alpha_i)\qtq{whenever}c_k<0.
\end{equation*}
We have
\begin{equation*}
(c_{k+i})= 
\begin{cases}
a_{k+1}\cdots a_{n-1}0\tilde\alpha_1\tilde\alpha_2\cdots&\text{if $k<n$ and $a_k=0$,}\\
\tilde\alpha_1\tilde\alpha_2\cdots&\text{if $k=n$,}\\
\tilde\alpha_{k-n+1}\tilde\alpha_{k-n+2}\cdots&\text{if $k>n$ and $\tilde\alpha_{k-n}=0$.}
\end{cases}
\end{equation*}
Since 
\begin{equation*}
\tilde\alpha_{j+1}\tilde\alpha_{j+2}\cdots\le \tilde\alpha_1\tilde\alpha_2\cdots\le \alpha_1\alpha_2\cdots
\end{equation*}
for $j=0$ and for all $j\ge 1$ satisfying $\tilde\alpha_j=0$,\footnote{In fact for all $j\ge 0$ by the results of \cite{KomLor2007} and \cite{DevKomLor2016}.}
it follows that 
\begin{equation*}
(c_{k+i})\le
\begin{cases}
a_{k+1}\cdots a_{n-1}0\alpha_1\alpha_2\cdots&\text{if $k<n$ and $a_k=0$,}\\
\alpha_1\alpha_2\cdots&\text{if $k\ge n$ and $c_k=0$.}
\end{cases}
\end{equation*}
We conclude by observing that 
\begin{equation*}
a_{k+1}\cdots a_{n-1}0\alpha_1\alpha_2\cdots\le \alpha_1\alpha_2\cdots\qtq{if}k<n\qtq{and}a_k=0
\end{equation*}
because the sequence $a_1\cdots a_{n-1}0\alpha_1\alpha_2\cdots$ is $p$-quasi-greedy.

It follows that
\begin{align*}
\frac{a_{p,q}(x)-a_{p,q}(x_n)}{x-x_n} 
&=\frac{\pi_q(a_1a_2\cdots)-\pi_q(a_1\cdots a_{n-1}0\tilde\alpha_1\tilde\alpha_2\cdots)}{\pi_p(a_1a_2\cdots)-\pi_p(a_1\cdots a_{n-1}0\tilde\alpha_1\tilde\alpha_2\cdots)} \\
&=\left(\frac{p}{q}\right)^{n-1}\frac{\pi_q(1a_{n+1}a_{n+2}\cdots)-\pi_q(0\tilde\alpha_1\tilde\alpha_2\cdots)}{\pi_p(1a_{n+1}a_{n+2}\cdots)-\pi_p(0\tilde\alpha_1\tilde\alpha_2\cdots)}\\
&\ge \left(\frac{p}{q}\right)^{n-1}\frac{\pi_q(10^{\infty})-q^{-1}\pi_q(\tilde\alpha_1\tilde\alpha_2\cdots)}{\pi_p(1^{\infty})}\\
&=\left(\frac{p}{q}\right)^{n-1}\frac{p-1}{q}\left(1-\pi_q(\tilde\alpha_1\tilde\alpha_2\cdots)\right)
\end{align*}
for each $n$.
The right-hand side tends to infinity as $n\to\infty$ because $r<q$, and hence
\begin{equation*}
\pi_q(\tilde\alpha_1\tilde\alpha_2\cdots)<\pi_r(\tilde\alpha_1\tilde\alpha_2\cdots)=1.
\end{equation*}
Therefore $D_-a_{p,q}(x)=\infty$. 

Finally we investigate the right differentiability of $a_{p,q}$ in a point $x\in [0,\frac{1}{p-1})$.

If $x=0$, then $a_p(x)=0^{\infty}$. 
If $x_n\searrow x$, then $(a_{n,i}):=a_p(x_n)$ starts with $0^m1$ where $m=m(n)\to\infty$ as $n\to\infty$. 
Therefore
\begin{align*}
\frac{a_{p,q}(x_n)-a_{p,q}(x)}{x_n-x} 
&=\frac{\pi_q(0^m1a_{n,m+2}a_{n,m+3}\cdots)}{\pi_p(0^m1a_{n,m+2}a_{n,m+3}\cdots)}
\ge \left(\frac{p}{q}\right)^m\frac{\pi_q(10^{\infty})}{\pi_p(1^{\infty})}
\end{align*}
for each $n$.
Letting $n\to\infty$ we conclude that
$d_+a_{p,q}(0)=\infty$.

If $b_p(x)$ has a last non-zero digit $b_n=1$, then using relation \eqref{32} we obtain that
\begin{equation*}
D_+a_{p,q}(x)=\lim_{y\searrow x}\frac{a_{p,q}(x)-a_{p,q}(y)}{x-y}=-\infty.
\end{equation*}
\medskip 

(ii) If $p\in (1,2]$ and $q>p$, then $b_{p,q}$ is increasing Proposition \ref{p31}  (i), and we may apply the Lebesgue differentiability theorem. 

\medskip 
(iii) If $p>2$, then $J_p$ is a null set, and $B_{p,q}$ is differentiable in each point $x\in [0,\frac{1}{p-1}]\setminus J_p$, because it is linear by definition in a neighborhood of $x$.
\end{proof}

In the following proposition we write $B_{p,q}$ instead of $b_{p,q}$ when $p\in (1,2]$, so that $B_{p,q}:[0,\frac{1}{p-1}]\to [0,\frac{1}{q-1}]$ for all $p,q>1$.

\begin{proposition}\label{p44} \mbox{} 

\begin{enumerate}[\upshape (i)]
\item $B_{p,q}$ is continuous if and only if $p>2$. 
\item $B_{p,q}$ has bounded variation if and only if $q\ge \min\set{p,2}$. 
\item $B_{p,q}$ is absolutely continuous if and only if $\min\set{p,q}>2$.
\item If $p>2$, then $B_{p,q}$ is  Hölder continuous with the exact exponent $\min\set{1,\frac{\log q}{\log p}}$.
\end{enumerate}
\end{proposition}

\begin{remark}\label{r45} 
Comparing Propositions \ref{p42}  and \ref{p44}  we see that in case $1<q<2<p$ the function $B_{p,q}$ is differentiable almost everywhere, although it has no bounded variation. 
This is due to the artificial linear extension over $J_p$.
\end{remark}

\begin{proof}[Proof of Proposition \ref{p44}]
(i) This follows from from Proposition \ref{p33}.
\medskip 

(ii) If $q\ge \min\set{p,2}$, then $B_{p,q}$ is even non-decreasing by Proposition \ref{p31} and by the affine nature of the extension defining $B_{p,q}$. 

Assume henceforth that $q<\min\set{p,2}$.
Fix a number $r$ such that $q<r<\min\set{p,2}$ and $(\beta_i):=b_r(1)$ has a last non-zero digit $\beta_m=1$.
This is possible by \cite[Lemma 3.1]{KomLor2007}.

Fix a positive integer $n$ and consider an arbitrary $r$-greedy sequence $b_1\cdots b_n\in\set{0,1}^n$. 
Then 
\begin{equation*}
b_1\cdots b_n0^{\infty}<b_1\cdots b_n0^m10^{\infty}
\end{equation*}
are $r$-greedy and hence also $p$-greedy sequences, so that
\begin{equation*}
\pi_p\left(b_1\cdots b_n0^{\infty}\right)<
\pi_p\left(b_1\cdots b_n0^m10^{\infty}\right).
\end{equation*}
On this interval the total variation of $B_{p,q}$ is at least 
\begin{equation*}
\pi_q\left(b_1\cdots b_n0^m10^{\infty}\right) -\pi_q\left(b_1\cdots b_n0^{\infty}\right) 
=\pi_q\left(0^{n+m}10^{\infty}\right) 
=\frac{1}{q^{n+m+1}}.
\end{equation*}

Now we recall from \cite[p. 490]{Ren1957} that there are at least $r^n$ $r$-greedy sequences $b_1\cdots b_n\in\set{0,1}^n$ for each $n$.
Hence the total variation of $B_{p,q}$ is at least 
\begin{equation*}
\sup_n\frac{r^n}{q^{n+m+1}}=\frac{1}{q^{m+1}}\left(\frac{r}{q}\right)^n=\infty.
\end{equation*}
\medskip 

(iii) In view of  (i) and (ii) it suffices to investigate the absolute continuity for $p>2$ and $q\ge 2$. 
Since in this case $B_{p,q}$ is non-decreasing by Proposition \ref{p31}, it is absolute continuous if and only if the Newton--Leibniz formula holds:
\begin{equation*}
\int_0^{1/(p-1)}B_{p,q}'(x)\ dx=B_{p,q}(1/(p-1))-B_{p,q}(0).
\end{equation*}

The following computation if valid for all  $p>2$ and $q>1$, and it also illustrates the non-monotonicity for $1<q<2$.

We are going to show that
\begin{equation*}
\int_0^{1/(p-1)}B_{p,q}'(x)\ dx=
\begin{cases}
1/(q-1)&\text{if $q>2$,}\\
0&\text{if $q=2$,}\\
-\infty&\text{if $1<q<2$.}
\end{cases}
\end{equation*}
Since
\begin{equation*}
B_{p,q}(1/(p-1))-B_{p,q}(0)=\frac{1}{q-1},
\end{equation*}
the Newton--Leibniz formula holds if and only if $q>2$.

Henceforth we write $f:=B_{p,q}$ for brevity.
We consider the connected components of $[0,\frac{1}{p-1}]\setminus J_p$ in the Cantor type construction of $J_p$, and we denote by $I_{m,k}$ the closures of these intervals for $m=0,1,\ldots$ and $k=1,\ldots,2^m$.
For example, 
\begin{equation*}
I_{0,1}:=\left[ \frac{1}{p(p-1)},\frac{1}{p}\right]
\end{equation*}
for $m=0$ and 
\begin{equation*}
I_{1,1}:=\left[ \frac{1}{p^2(p-1)},\frac{1}{p^2}\right],
\quad
I_{1,2}:=\left[\frac{1}{p}+ \frac{1}{p^2(p-1)},\frac{1}{p}+\frac{1}{p^2}\right]
\end{equation*}
for $m=1$.
It follows from the construction that the length of $I_{m,k}$ is equal to
\begin{equation}\label{41}
\abs{I_{m,k}}=\frac{p-2}{p^{m+1}(p-1)}
\end{equation}
for all $m,k$.

Since $B_{p,q}$ is affine on each interval $I_{m,k}$, the integral of $B_{p,q}$ of over $I_{m,k}$ may be computed by using the Newton--Leibniz formula.
For example,
\begin{equation}\label{42}
\begin{split}
\int_{I_{m,1}}B_{p,q}'(x)\ dx
&=\int_{\frac{1}{p^{m+1}(p-1)}}^{\frac{1}{p^{m+1}}}B_{p,q}'(x)\ dx\\
&=B_{p,q}\left(\frac{1}{p^{m+1}}\right)-B_{p,q}\left(\frac{1}{p^{m+1}(p-1)}\right)\\
&=\frac{1}{q^{m+1}}-\frac{1}{q^{m+1}(q-1)}\\
&=\frac{q-2}{q^{m+1}(q-1)}
\end{split}
\end{equation}
for all $m$.
By translation invariance the integrals over $I_{m,k}$ do not depend on $k$.
Therefore, taking into account that $J_p$ is a null set, we obtain that
\begin{equation*}
\int_0^{1/(p-1)}B_{p,q}'(x)\ dx
=\sum_{m=0}^{\infty}\sum_{k=1}^{2^m}B_{p,q}'(x)\ dx
=\sum_{m=0}^{\infty}2^m\frac{q-2}{q^{m+1}(q-1)}.
\end{equation*}

For $q=2$ each term, and hence the integral vanishes.
For $1<q<2$ the general term tends to $-\infty$, so that the integral is equal to $-\infty$.
For $q>2$ we have a convergent geometric series, and 
\begin{equation*}
\int_0^{1/(p-1)}B_{p,q}'(x)\ dx
=\frac{q-2}{q(q-1)}\sum_{m=0}^{\infty}\left(\frac{2}{q}\right)^m
=\frac{1}{1-\frac{2}{q}}\cdot\frac{q-2}{q(q-1)}=\frac{1}{q-1}.
\end{equation*}
\medskip 

(iv) The Hölder exponent of the extended function $B_{p,q}$ cannot be larger then the Hölder exponent $\alpha:=\frac{\log q}{\log p}$ of $b_{p,q}$, obtained in Proposition \ref{p33}, and it cannot be larger than $1$ because $B_{p,q}$ is defined on an interval. 
It remains to show that $B_{p,q}$ \emph{is} Hölder continuous with the exponent $\min\set{1,\alpha}$.

In case $q>p>2$  we have to prove that $B_{p,q}$ is Lipschitz continuous. 
Since is affine on each intervals $I_{m,k}$,  it follows from the estimates \eqref{41}, \eqref{42} and from the above mentioned translation invariance that
\begin{equation}\label{43}
B_{p,q}'(x)=\frac{q-2}{q^{m+1}(q-1)}\cdot \frac{p^{m+1}(p-1)}{p-2}
\end{equation}
on each $I_{m,k}$. 
Since $q>p>2$, all these derivatives are positive and uniformly bounded:
\begin{equation*}
0<\frac{q-2}{q^{m+1}(q-1)}\cdot \frac{p^{m+1}(p-1)}{p-2}\le \frac{p(p-1)(q-2)}{q(q-1)(p-2)}.
\end{equation*}
Since $J_p$ is a null set, this shows that the absolutely continuous function $B_{p,q}$ has an a.e. bounded derivative.

Next we assume that $p>\max\set{q,2}$.
We already know that $B_{p,q}$ is Hölder continuous on $J_p$ with the exponent $\alpha<1$ and some constant $c_1$. 

We claim that $B_{p,q}$ is also Hölder continuous on each interval $I_{m,k}$ with the same exponent $\alpha$ and with some constant $c_2$ independent of $m$ and $k$. 

Since 
\begin{equation*}
\abs{B_{p,q}(x)-B_{p,q}(y)}\le \frac{p^{m+1}(p-1)(q-2)}{q^{m+1}(q-1)(p-2)}\abs{x-y}
\end{equation*}
for all $x,y\in I_{m,k}$ by \eqref{43}, it suffices to find a constant $c_2$ satisfying
\begin{equation*}
\frac{p^{m+1}(p-1)(q-2)}{q^{m+1}(q-1)(p-2)}\abs{x-y}\le c_2\abs{x-y}^{\alpha}
\end{equation*}
for all $x,y\in I_{m,k}$, or equivalently that
\begin{equation*}
\frac{p^{m+1}(p-1)(q-2)}{q^{m+1}(q-1)(p-2)}\abs{I_{m,k}}^{1-\alpha}\le c_2
\end{equation*}
for all $m,k$.
Since $p^{\alpha}=q$ by the definition of $\alpha$, this is satisfied: 
\begin{align*}
\frac{p^{m+1}(p-1)(q-2)}{q^{m+1}(q-1)(p-2)}&\abs{I_{m,k}}^{1-\alpha}\\
&=\left( \frac{p}{q}\right)^{m+1}\cdot \frac{p-1}{q-1}\cdot \frac{q-2}{p-2}\left(\frac{p-2}{p^{m+1}(p-1)}\right)^{1-\alpha}  \\ 
&= \frac{q-2}{q-1}\left(\frac{p-1}{p-2}\right)^{\alpha}=:c_2. 
\end{align*}
The estimate 
\begin{equation*}
\abs{f(x)-f(y)}\le c_2\abs{x-y}^{\alpha}
\end{equation*}
remains valid by continuity on the closure of each interval $I_{m,k}$.

Finally we prove that 
\begin{equation*}
\abs{f(x)-f(y)}\le (c_1+2c_2)\abs{x-y}^{\alpha}
\end{equation*}
for all $x,y\in [0,\frac{1}{p-1}]$.

Since $[0,\frac{1}{p-1}]\setminus J_p$ is dense in $[0,\frac{1}{p-1}]$ (because $J_p$ is a null set), we may assume by continuity that $x,y\in [0,\frac{1}{p-1}]\setminus J_p$, and we may assume by symmetry that 
$x<y$.  

If $x$ and $y$ belong to the same connected component $I_{m,k}$, then we already know this inequality with $c_2$ in place of $c_1+2c_2$. 
Otherwise we have  $x\in I_{m,k}$ and $y\in I_{n,\ell}$ with different connected components. 
In this case we introduce the right endpoint $u$ of $I_{m,k}$, the left endpoint $v$ of $I_{n,\ell}$, and we conclude as follows: 
\begin{align*}
\abs{f(x)-f(y)} 
&\le \abs{f(x)-f(u)}+\abs{f(u)-f(v)}+\abs{f(v)-f(y)}\\
&\le c_2\abs{x-u}^{\alpha}+c_1\abs{u-v}^{\alpha}+c_2\abs{v-y}^{\alpha}\\ 
&\le (c_2+c_1+c_2)\abs{x-y}^{\alpha}.\qedhere
\end{align*}
\end{proof}

If $p>q=2$, then $B_{p,q}$ is non-decreasing, continuous, but not absolutely continuous by Propositions \ref{p31} and \ref{p44}.\footnote{We recall that Cantor's ternary function corresponds to the case $p=3$.}
Its arc length cannot be larger than $1+\frac{1}{p-1}$. 

Indeed, if, more generally, $f:[a,b]\to [c,d]$ is a non-decreasing function, then, following \cite{HilleTamarkin1929} for any finite subdivision $a=x_0<\cdots<x_n=b$ we have 
\begin{align*}
\sum_{i=1}^n&\sqrt{(x_i-x_{i-1})^2+(f(x_i)-f(x_{i-1}))^2} \\
&\qquad\le \sum_{i=1}^n\left((x_i-x_{i-1})+(f(x_i)-f(x_{i-1}))\right)  \\
&\qquad=(b-a)+(f(b)-f(a)) \\
&\qquad\le (b-a)+(d-c).
\end{align*}

In fact, for $B_{p,2}$ this maximum is achieved:

\begin{proposition}\label{p46} 
If $p>2$, then $B_{p,2}$ is continuous, non-decreasing and its arc length is equal to $p/(p-1)$.
\end{proposition}

\begin{proof}
We already know from Proposition \ref{p44} that $B_{p,2}$ is continuous and non-decreasing.
We approximate the arc length of $B_{p,2}$ by a sequence of polygonal lines as follows. 
We use the intervals $I_{m,k}$ from the preceding proof.

The length of the polygonal line determined by the endpoints of the intervals $[0,1/(p-1)]$ and  $I_{0,1}=[1/(p(p-1)),1/p]$ is equal to
\begin{equation*}
L_1:=\frac{p-2}{p(p-1)}+2\sqrt{\left( \frac{1}{p(p-1)}\right) ^2+\frac{1}{4}}.
\end{equation*}

If we add the four endpoints of the intervals $I_{1,1}$ and $I_{1,2}$, then we obtain the arclength
\begin{equation*}
L_2=\frac{p-2}{p(p-1)}+\frac{2(p-2)}{p^2(p-1)}+4\sqrt{\left( \frac{1}{p^2(p-1)}\right) ^2+\frac{1}{4^2}}.
\end{equation*}

At the next step we add the endpoints of four intervals $I_{2,1},\ldots, I_{2,4}$ to obtain the arclength
\begin{equation*}
L_3=\frac{p-2}{p(p-1)}
+\frac{2(p-2)}{p^2(p-1)}
+\frac{4(p-2)}{p^3(p-1)}
+8\sqrt{\left( \frac{1}{p^3(p-1)}\right) ^2+\frac{1}{4^3}}.
\end{equation*}

Continuing by induction  we obtain 
\begin{equation*}
L_n=2^n\sqrt{\left( \frac{1}{p^n(p-1)}\right) ^2+\frac{1}{4^n}}+\sum_{k=1}^{n}\frac{2^{k-1}(p-2)}{p^k(p-1)}
\end{equation*}
for $n=1,2,\ldots .$

Letting $n\to\infty$ we obtain that the arc length of $B_{p,2}$ is at least
\begin{equation*}
\lim L_n=1+\frac{p-2}{p(p-1)}\sum_{k=1}^{\infty}(2/p)^{k-1}=1+\frac{p-2}{p(p-1)}\cdot \frac{1}{1-(2/p)}=\frac{p}{p-1}.
\end{equation*}
In view of the preceding remark we have equality here.
\end{proof}

\end{document}